\documentclass[11pt]{amsart}

\usepackage[left=1in, right=1in, top=1.2in, bottom=1.2in]{geometry}
\usepackage{microtype, mathtools, mathrsfs, enumerate, dsfont, esvect}
\usepackage{verbatim}
\usepackage[linktocpage]{hyperref}
\hypersetup{
    colorlinks=true,        
    linkcolor=red,          
    citecolor=blue,         
    filecolor=magenta,      
    urlcolor=cyan           
}

\usepackage{amssymb}
\usepackage{url}

\newtheorem{theorem}{Theorem}[section]
\newtheorem{thm}{Theorem}[section]

\newtheorem{prop}[theorem]{Proposition}
\newtheorem{lemma}[theorem]{Lemma}

\newtheorem{cor}[theorem]{Corollary}

\theoremstyle{definition}

\theoremstyle{plain}
\numberwithin{equation}{theorem}

\theoremstyle{remark}
\newtheorem{claim}[theorem]{Claim}

\newcommand{\F}{{\mathbb F}}

\newcommand{\Q}{{\mathbb Q}}

\newcommand{\Z}{{\mathbb Z}}
\newcommand{\N}{{\mathbb N}}

\newcommand{\fp}{\mathfrak p}

\newcommand{\id}{\mathrm{id}}

\newcommand{\Qbar}{\overline{\Q}}

\DeclareMathOperator{\End}{End}

\newcommand{\isomto}{\overset{\sim}{\rightarrow}}

\newcommand{\bP}{{\mathbb P}}

\newcommand{\bG}{{\mathbb G}}

\newcommand{\lra}{\longrightarrow}

\newcommand{\cL}{\mathcal{L}}

\newif\ifhascomments \hascommentstrue
\ifhascomments
  \newcommand{\dragos}[1]{{\color{red}[[\ensuremath{\bigstar\bigstar\bigstar} #1]]}}
  \newcommand{\matt}[1]{{\color{red}[[\ensuremath{\spadesuit\spadesuit\spadesuit} #1]]}}
\else
  \newcommand{\dragos}[1]{}
  \newcommand{\matt}[1]{}
\fi

\begin{document}

\title[Periodic subvarieties and primitive ideals]{Periodic subvarieties of semiabelian varieties and annihilators of irreducible representations}

\author{Jason P. Bell}
\author{Dragos Ghioca}

\address{Department of Pure Mathematics\\
University of Waterloo\\
Waterloo, ON N2L 3G1\\
Canada}
\email{jpbell@uwaterloo.ca}

\address{Department of Mathematics \\
University of British Columbia\\
Vancouver, BC V6T 1Z2\\
Canada}
\email{dghioca@math.ubc.ca}

\begin{abstract}
Let $G$ be a semiabelian variety defined over a field of characteristic $0$, endowed with an endomorphism $\Phi$. We prove there is no proper subvariety $Y\subset G$ which intersects the orbit of each periodic point of $G$ under the action of $\Phi$.  As an application, we are able to give a topological characterization of the annihilator ideals of irreducible representations in certain skew polynomial algebras.
\end{abstract}



\thanks{The authors were partially supported by Discovery Grants from the National Sciences and Engineering Research Council of Canada.}

\maketitle


\section{Introduction}
\label{intro section}

\subsection{A sample of our results} 
In our paper we prove various connections between purely algebraic properties of rings and arithmetic dynamics on semiabelian varieties. In order to state one of our main results,  we first need to introduce the terminology of \emph{primitive ideals}, of \emph{skew polynomial rings}, and also of \emph{rational ideals}.

Given a simple left $R$-module $M$, the annihilator, ${\rm Ann}_R(M)$, of $M$ is the set of $x\in R$ such that $xM=(0)$.  This is a two-sided ideal and any ideal of this form is called a \emph{primitive} ideal.  To be pedantic, it is called a \emph{left primitive} ideal, with right primitivity of ideals being defined analogously.  Bergman \cite{Ber64} gave an example of a ring in which $(0)$ is right primitive but not left primitive, but his example is highly pathological and in practice, for well behaved classes of algebras, one has that the notions of left and right primitivity coincide and hence we omit the ``left'' when talking about primitive ideals, since it is well known that the two notions coincide in the context we consider in this paper.  As is customary, we say that a ring $R$ is \emph{primitive} if $(0)$ is a primitive ideal.  Also, we recall that given a ring $R$ and an automorphism $\sigma$ of $R$, one can form the \emph{skew polynomial} ring $R[x;\sigma]$, which is, as a set, $R[x]$, but with ``twisted'' multiplication given by $x\cdot r = \sigma(r) x$ for $r\in R$. One can similarly define the skew Laurent polynomial ring $R[x^{\pm 1};\sigma]$.  

Finally, in order to introduce the notion of rational ideals (see also Section~\ref{subsec:rings}), we note that given a left noetherian ring $R$ and a prime ideal $P$, Goldie (see \cite[Chapter 3]{Row1}) shows that one can localize $S=R/P$ at the set of elements that are not left or right zero divisors and one will obtain a quotient ring $Q(S)$, which is a simple Artinian ring.  One can intuitively think of $R/P$ as being a ``noncommutative field of fractions'' of $R$ and the Artin-Wedderburn theorem says that $Q(S)$ is isomorphic to a matrix ring over a division ring.  In the case that $k$ is a field and $R$ is a $k$-algebra, we say that a prime ideal $P$ of $R$ is \emph{rational} if the centre of $Q(R/P)$ is an algebraic extension of the base field.

In Section~\ref{sec:proofs 2}, we prove the following result.
\begin{thm} 
\label{thm:primitive}
Let $k$ be a field of characteristic zero, let $R=k[x_1^{\pm 1},\ldots ,x_d^{\pm 1}]$ and let $\sigma$ be a $k$-algebra automorphism of $R$.  Then the following results hold:
\begin{enumerate}
\item[(a)] a prime ideal $P$ of $S:=R[t;\sigma]$ is primitive if and only if $P$ is locally closed in ${\rm Spec}(S)$;
\item[(b)] a prime ideal $P$ of $T:=R[t^{\pm 1};\sigma]$ is primitive if and only if $P$ is rational.  
\end{enumerate}
\end{thm}

\subsection{The motivation for our problem from a ring theoretical point of view}
\label{subsec:rings}

One of the interesting developments that has come out of the evolution of noncommutative projective geometry is the relationship between ring theoretic properties of algebras produced from geometric data and dynamical phenomena involving algebraic varieties equipped with an automorphism.  
  In fact, these connections were in part what gave impetus to the first collaboration of the authors (along with Tom Tucker) \cite{BGT}, where one of the consequences of the authors' work was to prove a conjecture of Keeler, Rogalski, and Stafford \cite{KRS} concerning when algebras associated to the na\"ive blowup of an integral projective scheme at a closed point are noetherian.  In fact, many basic ring theoretic questions about twisted homogeneous coordinate rings and skew polynomial rings often lead to deep and highly non-trivial dynamical questions.  A recent example of this can be seen by the work of Brown, Carvalho, and Matczuk \cite{BCM17}, who studied the property of when skew polynomial rings of automorphism type have the property that the injective hulls of simple modules are locally Artinian.   In their paper, in addition to characterizing when this property holds for many skew polynomial rings, they use an argument supplied by Goodearl to fill a gap that occurs in a paper of Jordan \cite{Jor93} about skew polynomial algebras.   Jordan looked at a special example of skew rings in which $R=\mathbb{C}[x^{\pm 1},y^{\pm 1}]$ and $\sigma$ is the $\mathbb{C}$-algebra automorphism given by $\sigma(x)=y$ and $\sigma(y)=x^{-1}y$.  Jordan's goal was to show that the corresponding skew polynomial algebra has the property that every irreducible representation has a  non-trivial annihilator.  

Given an associative ring $R$, a fundamental problem is to classify the irreducible representations (simple $R$-modules) of $R$.  Understanding these representations yields a great deal of information about the ring itself and this representation theoretic approach to understanding rings has become an essential part of the study of representations of finite groups (and their group rings) and can be seen in the various local-global principles that arise in commutative algebra.  Unfortunately, it is generally an intractable problem to classify the simple modules of a general ring and in practice one often settles for a coarser understanding, proposed as an alternative by Dixmier, to instead understand annihilators of simple modules, which leads naturally to the aforementioned problem of characterizing primitive ideals in the prime spectrum.

We note that if $M$ is a simple $R$-module with annihilator $P$, then $M$ can be viewed as an $R/P$-module in a natural way and it becomes a faithful $R/P$-module; that is, ${\rm Ann}_{R/P}(M)=(0)$.  In this case, Jacobson's density theorem (see \cite[Chapter 2]{Row1}) says that $R/P$ embeds as a dense subring of an endomorphism ring of a $\Delta$-vector space, where $\Delta={\rm End}_{R/P}(M)$. For commutative rings $R$, a simple module is isomorphic to a cyclic module of the form $R/P$ with $P$ a maximal ideal.  In this case, $P$ is the annihilator so primitive ideals are precisely the maximal ideals of $R$, which are exactly the closed points in ${\rm Spec}(R)$.  In general, the primitive ideals form a subset of the prime ideals of $R$, where a two-sided ideal $I$ of $R$ is prime if whenever $JL\subseteq I$, with $J,L$ two-sided ideals of $R$, we have either $J\subseteq I$ or $L\subseteq I$.  The problem of determining the subset of primitive ideals of ${\rm Spec}(R)$ is an important part of understanding the representation theory of $R$.  In the case of enveloping algebras of finite-dimensional complex Lie algebras, Dixmier \cite{Dix77} and Moeglin \cite{Moe80}, gave a concrete description of the primitive spectrum in terms of the related properties of an ideal of being rational and also of being primitive.  Dixmier and Moeglin proved that for $P\in {\rm Spec}(U)$ with $U$ the enveloping algebra of a finite-dimensional complex Lie algebra, we have
$$P~{\rm primitive}\iff P~{\rm rational}\iff P~{\rm locally~closed~in}~{\rm Spec}(U).$$
Our Theorem~\ref{thm:primitive} is inspired in part by the results of Dixmier and Moeglin and proves analogues of their results for prime ideals in rings of the form $R[t;\sigma]$ or $R[t^{\pm 1};\sigma]$, where $R=k[x_1^{\pm 1},\dots, x_d^{\pm 1}]$.

\subsection{Connections to algebraic geometry and our results regarding arithmetic dynamics on semiabelian varieties}
\label{subsec:geometry}

When one gets into rings ``coming from geometry'', the problem of classifying primitive ideals often takes on a dynamical flavour.  One of the simplest classes of such rings are twisted polynomial rings $R[x;\sigma]$ with $R$ a finitely generated $k$-algebra and $\sigma:R\to R$ a $k$-algebra automorphism of $R$.  Here, the ring $R[x;\sigma]$ is just the polynomial ring $R[x]$ as a set, but with ``twisted'' multiplication given by $x\cdot r = \sigma(r)\cdot x$ for $r\in R$.  In this case, the primitive ideals were characterized by Leroy and Matczuk \cite{LM96}.
Given a ring $R$ and $\sigma\in {\rm Aut}(R)$, we say that an element $a$ of $R$ is $\sigma$-\emph{special} if for every $\sigma$-stable ideal $I$ of $R$ we have $a\sigma(a)\cdots \sigma^n(a)\in I$ for some $n$ and there is no $m\ge 1$ such that $a\sigma(a)\cdots \sigma^m(a)=0$.  Leroy and Matczuk \cite{LM96} show that if $R$ is a ring and $\sigma$ is an automorphism of $R$ then $R[x;\sigma]$ is primitive if and only if it has a $\sigma$-special element and $\sigma$ has infinite order.  When one looks at the case when $R$ is a finitely generated $k$-algebra that is an integral domain and $\sigma$ is a $k$-algebra automorphism, one has a corresponding automorphism $\tau$ of the irreducible affine variety $X={\rm Spec}(R)$. The $\sigma$-special property can then be translated as follows: $R$ is $\sigma$-special if and only if there is a proper closed subset $Y$ of $X$ such that every $\tau$-periodic subvariety of $X$ has an irreducible component contained in $Y$.  On the other hand, the property of $(0)$ being locally closed in the skew polynomial algebra is equivalent to the union of all proper periodic subvarieties being a proper Zariski closed set (see our Lemma \ref{lem:lc}).  It is natural to ask whether primitive ideals are precisely those prime ideals that are locally closed in the Zariski topology for algebras of the form $R[x;\sigma]$.  
In terms of arithmetic dynamics, if one has a quasiprojective variety $X$ equipped with an endomorphism $\Phi$, this is asking about the equivalence of the two properties:
\begin{enumerate}
\item[(I)] there is a proper subvariety $Y\subset X$ that contains all the proper, irreducible, periodic subvarieties of $X$ (under the action of $\Phi$). 
\item[(II)] there is a proper subvariety $Y\subset X$ that contains some iterate of each proper, irreducible, periodic subvariety of $X$.
\end{enumerate}
Clearly, property~(I) implies property~(II); the difficult part is to prove that (II) also implies (I).

In Section~\ref{sec:proofs} we prove the following results for semiabelian varieties (which are themselves extensions of abelian varieties by algebraic tori in the category of algebraic groups). 
 
\begin{thm}
\label{thm:endomorphisms}
Let $G$ be a semiabelian variety defined over an algebraically closed field $K$ of characteristic $0$ and let $\Phi$ be a dominant group endomorphism of $G$. Then there is no proper subvariety of $G$ which intersects the orbit of each torsion point that is periodic under the action of $\Phi$.
\end{thm}

Theorem~\ref{thm:endomorphisms} yields that neither property (I) nor property (II) holds for a dominant group endomorphism of a semiabelian variety. Also, we note that it suffices in Theorem~\ref{thm:endomorphisms} to restrict to orbits of \emph{torsion} points of $G$ that are periodic under the action of $\Phi$ and still derive that no proper subvariety of $G$ may intersect all these orbits. In particular, Theorem~\ref{thm:endomorphisms} allows us to derive the equivalence of properties~(I)~and~(II) for any regular self-map on a semiabelian variety, as established in the next result.

\begin{thm}
\label{thm:translations}
Let $G$ be a semiabelian variety defined over an algebraically closed field $K$ of characteristic $0$ and let $\Phi:G\lra G$ be a dominant, regular self-map. Then one of the following two statements must hold:
\begin{enumerate}
\item[(A)] there is a proper subvariety $X\subset G$ that contains all the proper, irreducible, periodic subvarieties of $G$ (under the action of $\Phi$). 
\item[(B)] there is no proper subvariety $Y\subset G$ that contains some iterate of each proper, irreducible, periodic subvariety of $G$.
\end{enumerate}
\end{thm}

In particular, Theorem~\ref{thm:translations} in the special case $G=\bG_m^n$ is used in deriving the conclusion in Theorem~\ref{thm:primitive}.

\subsection{Further connections with algebraic dynamics}

Besides the motivation from algebra (explained in Subsection~\ref{subsec:geometry}),  Theorems~\ref{thm:endomorphisms} and \ref{thm:translations} are also motivated by various questions in algebraic geometry and arithmetic dynamics, as we will explain below. 

If $\Phi$ is a polarizable endomorphism of a projective variety $X$ defined over a field of characteristic $0$ (i.e., there exists an ample line bundle $\cL$ on $X$ such that $\Phi^*\cL$ is linearly equivalent with $\cL^{\otimes d}$ for some $d>1$), then Fakhruddin \cite[Theorem~5.1]{Fakhruddin} proved that the periodic points are dense. Furthermore, a more careful analysis of the proof of \cite[Proposition~5.5]{Fakhruddin} yields the existence of a set $S$ of periodic points of $X$ with the property that choosing a point from the orbit of each point in $S$ would always yield a Zariski dense set in $X$; thus the conclusion in Theorem~\ref{thm:endomorphisms} holds for polarizable dynamical systems $(X,\Phi)$.

Next we discuss the connections between our results and the Dynamical Manin-Mumford conjecture. The original form of the Dynamical Manin-Mumford conjecture, formulated by Zhang in early 1990s (see \cite{Zhang}) predicts that for a polarizable dynamical system $(X,\Phi)$ defined over a field of characteristic $0$, if $V\subset X$ is a subvariety which contains a Zariski dense set of preperiodic points, then $V$ must be preperiodic itself. Later (see \cite{GTZ}), this conjecture was ammended; however, it is expected that a variant of the Dyanmical Manin-Mumford Conjecture holds for more general dynamical systems that are not necessarily polarizable (see \cite{DF, GNY-1, GNY-2} for similar results). So, asssume that for the endomorphism $\Phi$ of a a quasiprojective variety $X$, the aforementioned Dynamical Manin-Mumford Conjecture holds; we claim that this yields that a variant of the conclusions from Theorems~\ref{thm:endomorphisms}~and~\ref{thm:translations} must hold. Indeed, assume there is a proper subvariety $Y\subset X$ which contains some iterate of each periodic point of $X$. Without loss of generality, we may assume $Y$ is the Zariski closure of a set of periodic points; then the Dynamical Manin-Mumford Conjecture yields that $Y$ must be preperiodic and therefore, its orbit under $\Phi$ is a closed, proper, subvariety $Z$ of $X$. Then our assumption on $Y$ yields that $Z$ must contain \emph{all} periodic points of $X$.

Our Theorems~\ref{thm:endomorphisms} and \ref{thm:translations} (and especially,  their proofs) are also motivated by the special case of the Medvedev-Scanlon conjecture for regular self-maps on semiabelian varieties. The Medvedev-Scanlon conjecture from \cite{MS-Annals} is based on a much earlier conjecture of Zhang from the early 1990s stated in \cite{Zhang}; the Medvedev-Scanlon conjecture  predicts that for a dominant rational self-map $\Phi$ of a quasiprojective variety $X$ defined  over an algebraically closed field $K$ of characteristic $0$, either there exists a point $x\in X(K)$ with a Zariski dense orbit, or there exists a non-constant rational function $f:X\dashrightarrow \bP^1$ such that $f\circ \Phi=f$. This conjecture is known to hold when $K$ is uncountable (see \cite{A-C, dynamical Rosenlicht}), but it is very difficult in the case of countable fields (for various results in this case, see \cite{MS-Annals, 3-folds, GH, GS-abelian, GS-semiabelian, GX}). Similar to  the proof of the Medvedev-Scanlon conjecture when $X$ is a semiabelian variety (see \cite{GS-semiabelian}, which extends the results of \cite{GS-abelian}), in our arguments deriving Theorem~\ref{thm:translations} we reduce to the special case $\Phi$ is a composition of a translation with a unipotent endomorphism (this last case bears resemblance also to the study of wild automorphisms for abelian varieties from \cite{wild automorphism}). In Lemma~\ref{lem:red Jor} we prove that given a dynamical system $(X,\Phi)$, if there exists a proper subvariety $Y\subset X$ containing an iterate of each proper, irreducible, periodic subvariety of $X$, then there must exist an algebraic point of $X$ with a Zariski dense orbit under $\Phi$. Our results allow us to simplify a characterization of Jordan for the primitivity of skew Laurent polynomial algebras over affine commutative domains (see Proposition~\ref{rem:Jordan}).


\section{Proofs of Theorems~\ref{thm:endomorphisms} and \ref{thm:translations}}
\label{sec:proofs}

We start by proving some useful reductions for the proof of  Theorem~\ref{thm:endomorphisms}; similar statements hold also for Theorem~\ref{thm:translations}. The first statement is a simple, but useful observation, which we will use throughout the entire proof of Theorem~\ref{thm:endomorphisms}. We also note that throughout this section, our semiabelian variety $G$ is always assumed to be defined over an algebraically closed field $K$ of characteristic $0$.

\begin{lemma}
\label{lem:torsion-periodic}
Let $G$ be a semiabelian variety endowed with a group endomorphism $\Phi$. 
\begin{enumerate}
\item[(a)] Then the orbit of each periodic, torsion point of $G$ contains only torsion points.
\item[(b)] Each torsion point is preperiodic under the action of $\Phi$.
\item[(c)] Then $Y\subset G$ intersects the orbit of each periodic, torsion point if and only if $Y$ intersects the orbit of each torsion point.
\end{enumerate}
\end{lemma}

\begin{proof}
The proof is immediate, by noting that $\Phi$ commutes with the multiplication-by-$m$ map (always denoted by $[m]$) on $G$, for any integer $m$. In particular, letting $G[m]$ be the (finite) set of torsion points of $G$ killed by $[m]$, we see that $\Phi(G[m])\subseteq G[m]$ (which justifies both (a) and (b)). Finally, part~(c) follows from part~(b).
\end{proof}

Our next observation is also simple and useful.
\begin{lemma}
\label{lem:iterate}
The conclusion of Theorem~\ref{thm:endomorphisms} is unchanged if we replace the dynamical system $(G,\Phi)$ by the dynamical system $(G,\Phi^\ell)$ for some positive integer $\ell$.
\end{lemma}

\begin{proof}
First of all, it is clear that the two dynamical systems share the same set of periodic points; therefore, the conclusion of Theorem~\ref{thm:endomorphisms} for $(G,\Phi)$ yields the same conclusion for $(G,\Phi^\ell)$. Conversely, if there were some proper subvariety $Y\subset G$ containing some iterate of each periodic point of $G$ under the action of $\Phi$, then $\cup_{i=0}^{\ell-1}\Phi^i(Y)$ is also a proper subvariety of $G$, which would then contain some iterate of each periodic point under the action of $\Phi^\ell$.
\end{proof}

The next result shows that we can always replace the dynamical system $(G,\Phi)$ by an isogenous copy of it.

\begin{lemma}
\label{claim:red 1 abelian}
Theorem~\ref{thm:endomorphisms} is invariant if we replace $G$ by an isogenous semiabelian variety $G'$. 
\end{lemma}

\begin{proof}
Let $\tau:G\lra G'$ be an isogeny, and assume Theorem~\ref{thm:endomorphisms} holds for any dominant  group endomorphism of $G'$; we show next that it must hold also for any dominant group endomorphism of $G$.

Let $\Phi$ be a group endomorphism of $G$, and let $m$ be a positive integer such that $\ker(\tau)\subset G[m]$. We let $\Phi_m:=\Phi\circ [m]$ and note that $\Phi_m':=\tau\circ \Phi_m\circ\tau^{-1}$ is a well-defined, dominant group endomorphism of $G'$. Since $\tau\circ \Phi_m=\Phi_m'\circ \tau$, we obtain that Theorem~\ref{thm:endomorphisms} holds for the dynamical system $(G,\Phi_m)$ if and only if it holds for the dynamical system $(G', \Phi_m')$; note that if there were a proper subvariety $Y\subset G$ intersecting the orbit of each torsion point of $G$ (see Lemma~\ref{lem:torsion-periodic}), then $\tau(Y)$ is a proper subvariety of $G'$ intersecting the orbit of each  torsion point of $G'$ (and also the converse holds since $\tau$ is a finite map). Therefore, it remains to prove that if Theorem~\ref{thm:endomorphisms} holds for $(G,\Phi_m)$, then it must hold also for $(G,\Phi)$.

Assume that Theorem~\ref{thm:endomorphisms} fails for the dynamical system $(G,\Phi)$; then there exists some proper subvariety $Y\subset G$ intersecting the orbit of each torsion point of $G$ (see Lemma~\ref{lem:torsion-periodic}~(c)).  Without loss of generality we may assume $Y$ is the Zariski closure of a subset of torsion points (see Lemma~\ref{lem:torsion-periodic}~(a)). Thus, by Laurent's theorem \cite{Laurent}, we get that $Y$ is a finite union of torsion translates of proper algebraic subgroups of $G$. In particular, this yields that 
$Y_m:=\bigcup_{k=1}^\infty [m^k](Y)$ (i.e., the union of the image of $Y$ under the multiplication-by-$m^k$ morphisms, as we let $k$ vary) is also a proper, Zariski closed subset of $G$. Furthermore, by construction (and also using the hypothesis regarding $Y$), we get that $Y_m$ intersects the orbit of each  torsion point of $G$ under the action of $\Phi_m$. This contradicts the fact that conclusion of Theorem~\ref{thm:endomorphisms} must hold for the dynamical system $(G,\Phi_m)$, thus proving the desired reduction from Lemma~\ref{claim:red 1 abelian}. 
\end{proof}

The next result shows in particular that proving Theorem~\ref{thm:endomorphisms} reduces to deriving the desired conclusion in the case of abelian varieties and respectively, in the case of algebraic tori (note that each semiabelian variety is an extension of an abelian variety by a torus).

\begin{lemma}
\label{claim:red 3 semiabelian}
Let $G$ be a semiabelian variety endowed with a dominant group endomorphism $\Phi$ and assume there exists a short exact sequence of semiabelian varieties:
\begin{equation}
\label{eq:short exact 0}
1\lra G_1\lra G\lra G_2\lra 1
\end{equation}
such that $\Phi|_{G_1}$ is an endomorphism of $G_1$, while $\bar{\Phi}$ is the corresponding induced endomorphism of $G_2$. If the conclusion in Theorem~\ref{thm:endomorphisms} holds for the dynamical systems $\left(G_1,\Phi|_{G_1}\right)$ and $(G_2,\bar{\Phi})$, then the same conclusion from Theorem~\ref{thm:endomorphisms} holds for the dynamical system $(G,\Phi)$.
\end{lemma}

\begin{proof}
Let $Y\subseteq G$ be a subvariety which contains a point from the orbit of each periodic torsion point of $G$ under the action of $\Phi$. 
For each point $x\in G$, we let $\bar{x}\in G_2$ be its image under the morphism $G\lra G_2$ from \eqref{eq:short exact 0}. 

Let $x\in G$ be a torsion point that is periodic under the action of $\Phi$; let $N(x)\in\N$ be the length of the period of $x$. Then for any torsion point $y\in G_1$ that is periodic under the action of $\Phi|_{G_1}$, we have that also $x+y$ is both torsion and periodic under the action of $\Phi$ (since $\Phi$ is a group endomorphism). Our hypothesis regarding $Y$ yields that for each such point $y\in G_1$, we have that there exists some positive integer $n_y$ such that $\Phi^{n_y}(x)+\Phi^{n_y}(y)\in Y$. Our hypothesis regarding the dynamical system $(G_1,\Phi|_{G_1})$ yields that the set $\left\{\Phi^{n_y}(y)\right\}_y$ is Zariski dense in $G_1$. Thus, there exists some positive integer $N_0(x)\le N(x)$ with the property that the set
$$\left\{\Phi^{n_y}(y)\colon n_y\equiv N_0(x)\pmod{N(x)}\right\}$$
is Zariski dense in $G_1$. So, $\Phi^{N_0(x)}(x)+G_1\subset Y$ for each torsion point $x$ that is periodic under the action of $\Phi$. Using the fact that the points $\{\overline{\Phi^{N_0(x)}(x)}\}_x$ are Zariski dense in $G_2$ (according to the hypothesis applied to the dynamical system $(G_2,\bar{\Phi})$), we conclude that $Y=G$, as desired.  
\end{proof}

A special case of Lemma~\ref{claim:red 3 semiabelian} yields the following statement. 
\begin{cor}
\label{claim:red 2 abelian}
If the conclusion in our Theorem~\ref{thm:endomorphisms} holds for the semiabelian varieties $G_1$ and $G_2$ equipped with dominant group endomorphisms $\Phi_1$, respectively $\Phi_2$, then the same conclusion holds for the dynamical system $(G_1\times G_2, \Phi)$, where $\Phi$ is the endomorphism of $G:=G_1\times G_2$ defined by $\Phi(x_1,x_2)=(\Phi_1(x_1),\Phi_2(x_2))$. 
\end{cor}

We continue by proving Theorem~\ref{thm:endomorphisms} for powers of simple semiabelian varieties (i.e., $G$ is either isomorphic to some power of $\bG_m$, or to some power of a simple abelian variety); this special case is instrumental  in deriving the general case in Theorem~\ref{thm:endomorphisms}.

\begin{thm}
\label{thm:endomorphisms simple power}
Theorem~\ref{thm:endomorphisms} holds when $G=A^m$ for a simple semiabelian variety $A$ (i.e., there is no proper semiabelian subvariety of $A$).
\end{thm}

We split our proof of Theorem~\ref{thm:endomorphisms simple power} based on whether $A$ is isomorphic to the multiplicative group $\bG_m$, or $A$ is a simple  abelian variety.

\begin{proof}[Proof of Theorem~\ref{thm:endomorphisms simple power} when $A\isomto \bG_m$.]
Suppose there exists a proper subvariety $Y$ of $G$ (which we assume to be the $m$-th cartesian power of $\bG_m$) intersecting the orbit of each torsion point of $G$ that is periodic under the action of $\Phi$. In particular, we may assume $Y\subset G$ is the Zariski closure of a set of torsion points and thus, by Laurent's theorem \cite{Laurent}, we have that $Y$ is a finite union of torsion translates of algebraic subgroups of $G$.

We argue by induction on $m$; the case $m=1$ is obvious since there exist infinitely many torsion points of $\bG_m$ (which are the usual roots of unity in  $K$).

Because $Y$ is a finite union of torsion translates of algebraic subgroups of $A^m$, then there must exist finitely many $m$-tuples $(\alpha_{i,1},\dots,\alpha_{i,m})\in \Z^m\setminus\{(0,\dots,0)\}$ (for some $i=1,\dots, k$) such that each point $(x_1,\dots, x_m)\in G$ of $Y$ satisfies an equation of the form:
\begin{equation}
\label{eq:i}
\prod_{j=1}^m x_j^{\alpha_{i,j}}=1,
\end{equation}
for some $i=1,\dots, k$.

We let $M_\Phi$ be the $m$-by-$m$ matrix with integer entries  which corresponds to the endomorphism $\Phi$ of $G$. Since $\Phi$ is dominant, we have that all eigenvalues of $M_\Phi$ are nonzero. We let $L$ be a finite Galois extension of $\Q$ containing all the eigenvalues of $M_\Phi$ and also, containing  the entries of all the eigenvectors of $M_\Phi$. Then  there exists an infinite set $S$ of primes $p$ satisfying the following properties: 
\begin{itemize} 
\item[(i)] $p$ splits completely in $L/\Q$; and 
\item[(ii)] fixing some prime $\fp$ of $L$ lying above $p$, we have that each eigenvalue of $M_\Phi$ is a $\fp$-adic unit. 
\end{itemize}

Fix some $p\in S$. We let $\epsilon_p\in \Qbar$ be a root of unity of order precisely $p$. Our hypotheses~(i)--(ii) yield the existence of a vector $v:=(c_1,\dots, c_m)\in \Z^m$ and also the existence of an integer $b$ coprime with $p$  such that 
\begin{equation}
\label{eq:congruent p}
M_\Phi\cdot v\equiv b\cdot v\pmod{p}.
\end{equation}
Note that the congruence equation \eqref{eq:congruent p} first holds modulo $\fp$, but then since $M_\Phi$ and also the vector $v$ have all their entries integral, then we obtain a congruence modulo $p$. 
We let $\overrightarrow{x_p}:=\left(\epsilon_\fp^{c_1},\cdots, \epsilon_\fp^{c_m}\right)\in G(\Qbar)$. Then \eqref{eq:congruent p} yields that 
\begin{equation}
\label{eq:iteration p}
\Phi(x_p)=\overrightarrow{x_p}^b,
\end{equation}
where for each vector $\overrightarrow{v}:=(v_1,\dots, v_m)\in \Qbar^m$, we let $\overrightarrow{v}^b:=\left(v_1^b,\dots, v_m^b\right)$. Thus $\Phi^j(x_p)=\overrightarrow{x_p}^{b^j}$; since $\gcd(b,p)=1$, we get that $\overrightarrow{x_p}$ is periodic under the action of $\Phi$. Furthermore, our assumption on $Y$ yields the existence of some $j\in\N$ and also the existence of some $i_p\in \{1,\dots,k\}$ such that  
$$\prod_{\ell=1}^m \epsilon_p^{\alpha_{i_p,\ell}c_\ell b^j}=1;$$ 
hence (because $b$ is coprime with $p$, while $\epsilon_p$ has order $p$), we have 
\begin{equation}
\label{eq:congruence p c}
c_1\alpha_{i_p,1}+\cdots + c_m\alpha_{i_p,m}\equiv 0\pmod{p}.
\end{equation}
By the pigeonhole principle, there exists some $\tilde{i}\in\{1,\dots, k\}$ such that $i_p=\tilde{i}$ for infinitely many primes $p\in S$. We let $w:=(\alpha_{\tilde{i},1},\dots, \alpha_{\tilde{i},m})\in \Z^m$. Then, using \eqref{eq:congruence p c} and \eqref{eq:congruent p}, we have that 
\begin{equation}
\label{eq:dot product}
w\cdot M_\Phi^nv\equiv 0\pmod{p},
\end{equation}
for all nonnegative integers $n$ (where $w_1\cdot w_2$ represents the dot product of the two vectors $w_1$ and $w_2$). Note that in equation \eqref{eq:dot product}, the vector $v$ changes with the prime $p$, but the vector $w$ is unchanged.

Consider the $L$-vector space $U$ spanned by all vectors $(M_\Phi^t)^nw$ for $n\ge 0$ (where $M^t$ is the transpose of the matrix $M$). If $\dim U = m$, then equation \eqref{eq:dot product} cannot hold for infinitely many primes $p$ since for large primes $p$, there is no nonzero vector $v$ orthogonal modulo $p$ to each element in $U$.

Thus it must be that $\dim U<m$. In this case, we let $A_1$ be the algebraic subgroup of $G$ corresponding to the orthogonal complement of the linear subspace $U$ of the tangent space of $G$ at the origin; more precisely, $A_1$ is the algebraic subgroup of $G$ consisting of all $x:=(x_1,\dots, x_m)$ such that for each $n$, we have $[w](\Phi^n(x))=0$, where $[w]:G\lra \bG_m$ is given by the following formula (note that $w=(\alpha_{\tilde{i},1},\dots, \alpha_{\tilde{i},m})$):
$$(z_1,\dots, z_m)\mapsto \prod_{j=1}^m z_j^{\alpha_{\tilde{i},j}}.$$ 
Clearly, $A_1$ is fixed by $\Phi$; furthermore, since $\dim(U)<m$, we get that $A_1$ is a proper algebraic subgroup of $A^m$ (also note that not all entries of  $w$ are equal to $0$). Furthermore, an iterate of $\Phi$ restricts to an endomorphism of $A_1^0$, which is the connected component of $A_1$; note that $A_1^0$ is isomorphic to the $m_1$-st cartesian power of $\bG_m$ (for some  $0<m_1<m$). Since the conclusion of Theorem~\ref{thm:endomorphisms} is unchanged if we replace $\Phi$ by an iterate of it (see Lemma~\ref{lem:iterate}), without loss of generality, we assume $\Phi$ induces an endomorphism of $A_1^0$.  Thus we have a short exact sequence of algebraic groups: 
$$1\lra A_1^0\lra G\lra A_2\lra 1,$$
where $A_2$ is isomorphic to $\bG_m^{m_2}$ for some integer $0<m_2<m$. Moreover, $\Phi$ induces also an endomorphism $\bar{\Phi}$ of $A_2$ (because $\Phi$ restricts to an endomorphism of $A_1^0$).  Applying the inductive hypothesis to the dynamical systems $(A_1^0,\Phi|_{A_1^0})$ and $(A_2,\bar{\Phi})$ combined with Lemma~\ref{claim:red 3 semiabelian} yields the desired conclusion in Theorem~\ref{thm:endomorphisms} in the case $G$ is the cartesian power of the multiplicative group.
\end{proof}

\begin{proof}[Proof of Theorem~\ref{thm:endomorphisms simple power} when $A$ is an abelian variety.] 
We know that $A$ is an abelian variety of dimension $g\ge 1$. We start by proving some easy facts regarding endomorphisms of abelian varieties.

\begin{lemma}
\label{lem:endomorphism}
Let $\Psi$ be a dominant group endomorphism of $A^\ell$ (for some positive integer $\ell$). Then for all but finitely many primes $p$, we have that for each $n\in\N$, the endomorphism $\Psi^n$ of $A^\ell$ induces a bijection on $A[p]^\ell$ (where $A[p]$ is the $p$-torsion subset of $A$).
\end{lemma}

\begin{proof}[Proof of Lemma~\ref{lem:endomorphism}.]
We let $g\in\Z[t]$ be a monic polynomial of minimal degree, which kills the endomorphism $\Psi$ of $A^\ell$, i.e., $g(\Psi)=0$; for more details, see \cite[Fact~3.3]{GS-abelian} and \cite[Fact~2.6]{GS-semiabelian}. Since $\Psi$ is a dominant endomorphism, we have that $c:=g(0)$ is a nonzero integer. So, there exists a polynomial $h\in \Z[t]$ such that $h(\Psi)\cdot \Psi=[c]_\ell$, where $[c]_\ell$ is the coordinatewise multiplication-by-$c$ morphism on $A^\ell$. Clearly, for each prime $p$, which does not divide $p$, we have that  the endomorphism $[c]_\ell$ of $A^\ell$ induces a bijection on $A[p]^\ell$. In particular, $(h(\Psi)\cdot \Psi)|_{A[p]^\ell}$ is a bijection, which forces $\Psi|_{A[p]^\ell}$ be a bijection on $A[p]^\ell$; then $\Psi^n$, for any positive integer $n$, induces a bijection of $A[p]^\ell$, as claimed.
\end{proof}

\begin{lemma}
\label{lem:sum coprime}
Let $p$ be a prime number and let $\Psi_1,\dots, \Psi_\ell$ be group endomorphisms of $A$ with the property that there exists some $i_1\in\{1,\dots, \ell\}$ such that $\Psi_{i_1}$ induces a bijection on $A[p]$. Then there are exactly $p^{2g(\ell-1)}$ points $(x_1,\dots, x_\ell)\in A[p]^\ell$ such that 
\begin{equation}
\label{eq:psi}
\Psi_1(x_1)+\cdots + \Psi_\ell(x_\ell)=0.
\end{equation}
\end{lemma}

\begin{proof}[Proof of Lemma~\ref{lem:sum coprime}.]
For each $j\in\{1,\dots, \ell\}\setminus\{i_1\}$, we let $x_j$ be an arbitrary point in $A[p]$; this yields precisely $p^{2g(\ell-1)}$ such $(\ell-1)$-tuples. For each such $(\ell-1)$-tuple, there is exactly one solution $x_{i_1}\in A[p]$ such that equation \eqref{eq:psi} is satisfied, as claimed.
\end{proof}

We continue our proof arguing again by induction on $m$, similar to the case $G$ is a power of the multiplicative group; also, as before, the case $m=1$ is obvious since there exist infinitely many torsion points and $A$ has no proper algebraic subgroups because it is a simple abelian variety. We let $D:=\End(A)$.  Furthermore, as before, we assume there exists a proper subvariety $Y$ of $G$ intersecting the orbit of each torsion point of $G$. Moreover, we may assume $Y\subset G$ is the Zariski closure of a set of torsion points and thus, by Laurent's theorem \cite{Laurent}, we have that $Y$ is a finite union of torsion translates of algebraic subgroups of $G$. Thus there must exist finitely many $m$-tuples $(\alpha_{i,1},\dots,\alpha_{i,m})\in D^m\setminus\{(0,\dots,0)\}$ (for some $i=1,\dots, k$) such that each point $(x_1,\dots, x_m)\in A^m$ of $Y$ satisfies an equation of the form:
\begin{equation}
\label{eq:i 222}
[\alpha_{i,1}](x_1)+\cdots +[\alpha_{i,m}](x_m)=0,
\end{equation}
for some $i=1,\dots, k$ (where $[\alpha]$ is the endomorphism of $A$ represented by the element $\alpha\in D$). Also, we let $M_\Phi\in M_{m,m}(D)$ representing the endomorphism $\Phi\in\End(A^m)$.

Let $f\in \Z[t]$ be the minimal, monic polynomial of the endomorphism $\Phi$; i.e., $f(\Phi)=0$. Next we prove that we may assume that $f$ is irreducible.

\begin{claim}
\label{claim:irreducible}
If $f(t)$ is not irreducible, then the conclusion in Theorem~\ref{thm:endomorphisms simple power} follows from the inductive hypothesis. 
\end{claim}

\begin{proof}[Proof of Claim~\ref{claim:irreducible}.]
Assume $f(t)=g(t)\cdot h(t)$ for some non-constant polynomials $g,h\in\Z[t]$. We let $A_1:=\ker(g(\Phi))$ and also, let $A_2:=g(\Phi)(A^m)$. Then $A_2$ is an abelian variety (isogenous to $A^r$ for some integer $r<m$ since $A$ is a simple abelian variety), while $A_1$ is an algebraic subgroup of $A^m$ (not necessarily connected); however, the connected component $A_1^0$ of $A_1$ is isogenous to $A^{m-r}$. Then we have a short exact sequence of algebraic groups:
$$1\lra A_1\lra A^m\lra A_2\lra 1;$$
furthermore, $\Phi$ restricts to a group endomorphism of both $A_1$ and of $A_2$. Moreover, a suitable iterate $\Phi^\ell$ (for some $\ell\in\N$) induces an endomorphism of $A_1^0$ as well. Then, using Lemma~\ref{lem:iterate}, and also applying the inductive hypothesis to each dynamical system $(A_1^0,\Phi^\ell)$ and $(A_2,\Phi)$, coupled with Lemma~\ref{claim:red 3 semiabelian} yields the desired conclusion for the dynamical system $(A^m,\Phi)$, as claimed.   
\end{proof}

From now on, we may assume that the polynomial $f(t)$ is irreducible; in particular, this means that all its roots are distinct (and nonzero, since $\Phi$ is a dominant group endomorphism). Let $L$ be the splitting field for the polynomial $f(t)$; so, $L$ is a finite Galois extension of $\Q$ containing all the roots of $f$. Then there exists an infinite set $S$ of primes $p$, which split completely in $L/\Q$. Furthermore, at the expense of excluding finitely many primes $p$ from the infinite set $S$, we may even assume that:
\begin{enumerate}
\item[(A)] the prime $p$ and the endomorphism $\Phi$ of $A^m$ satisfy the conclusion of Lemma~\ref{lem:endomorphism}; and
\item[(B)] for each $i=1,\dots, k$, there exists an endomorphism  $[\alpha_{i,j(i)}]$ (for some $j(i)\in\{1,\dots, m\}$) which induces a bijection on $A[p]$.
\end{enumerate}
Note that Lemma~\ref{lem:endomorphism} yields that only finitely many primes $p$ do not satisfy condition~(A) above. Similarly, since for each $i=1,\dots, k$, we know that there is some $j(i)\in\{1,\dots, m\}$ such that $\alpha_{i,j(i)}\ne 0$, then Lemma~\ref{lem:endomorphism} yields that for all but finitely many primes $p$, the endomorphism $[\alpha_{i,j(i)}]$ of $A$ induces a bijection of $A[p]$.

Let $p\in S$. Since $f$ has distinct, nonzero roots and $p$ splits completely in $L/\Q$, we get that the reduction of $f(t)$ modulo $p$ divides the polynomial $t^{p-1}-1$ in $\F_p[t]$; thus there exists some polynomials $g_p,h_p\in\Z[t]$ such that
\begin{equation}
\label{eq:the endomorphism equation}
t^{p-1}-1=f(t)\cdot g_p(t) + p\cdot h_p(t).
\end{equation}
Hence $\Phi^{p-1}(x)=x+ph_p(\Phi(x))=x+h_p(\Phi(px))$ for each $x\in G=A^m$. Equation \eqref{eq:the endomorphism equation} yields that for each $x\in G[p]$, we have
\begin{equation}
\label{eq:p-1}
\Phi^{p-1}(x)=x.
\end{equation}

Let $x\in G[p]$. By our assumption, there exists some positive integer $n(x)$ such that $\Phi^{n(x)}(x)\in Y$; using \eqref{eq:p-1}, we see that we may assume $n(x)\in\{1,\dots, p-1\}$. We let $x=:(x_1,\dots, x_m)\in A[p]^m$; so, there exists some $i(x)\in\{1,\dots, k\}$ such that 
\begin{equation}
\label{eq:vanishing}
\sum_{j=1}^m \left[\overrightarrow{\alpha_{i(x)}}\right]\cdot \Phi^{n(x)}(x)=0,
\end{equation}
where for any $\delta_1,\dots, \delta_m\in D$ we define the morphism $[\delta_1, \dots, \delta_m]:A^m\lra A$ as follows: 
$$[\delta_1,\dots,\delta_m]\cdot x:=[\delta_1](x_1)+\cdots +[\delta_m](x_m).$$

\begin{lemma}
\label{lem:counting i}
Let $p\in S$, let $i\in\{1,\dots, k\}$ and let $n\in\N$. Then there are $p^{2g(m-1)}$ points $x\in A[p]^m$ such that $\left[\overrightarrow{\alpha_i}\right]\cdot \Phi^n(x)=0$.
\end{lemma}

\begin{proof}[Proof of Lemma~\ref{lem:counting i}.]
By our choice of the prime $p$ satisfying condition~(A) above, we get that $\Phi^n$ induces a bijection of $A[p]^m$. Then condition~(B), along with Lemma~\ref{lem:sum coprime} finishes the proof of Lemma~\ref{lem:counting i}.
\end{proof}

Lemma~\ref{lem:counting i} yields that there are at most $k(p-1)p^{2g(m-1)}<kp^{2gm-1}$ tuples $x:=(x_1,\dots, x_m)\in A[p]^m$, which satisfy the equation $[\overrightarrow{\alpha_i}]\cdot \Phi^n(x)=0$ for some $i=1,\dots, k$ and for some $n=1,\dots, p-1$. On the other hand, our hypothesis on the points of $Y$ yields that for each point $x\in A[p]^m$ there exists some $i(x)\in\{1,\dots, k\}$ and some $n(x)\in\{1,\dots, p-1\}$ satisfying equation \eqref{eq:vanishing}. This yields an inequality $p^{2gm}<kp^{2gm-1}$, which is false when $p>k$. This contradiction concludes our proof of Theorem~\ref{thm:endomorphisms simple power} when $G$ is a power of an abelian variety $A$. 
\end{proof}

Using the proof of Theorem~\ref{thm:endomorphisms} in the case $G$ is a power of a simple abelian variety (along with Lemmas~\ref{claim:red 1 abelian} and \ref{claim:red 3 semiabelian}), we can derive the same conclusion when $G$ is any abelian variety.

\begin{proof}[Proof of Theorem~\ref{thm:endomorphisms} when $G$ is an abelian variety.]
By Poincar\'e's theorem (see, for example, \cite[Fact~3.2]{GS-abelian}), $G$ is isogenous with a direct product of simple abelian varieties; more precisely, $G$ is isogenous to a product of the form $\prod_{i=1}^\ell A_i^{m_i}$, where each $A_i$ is a simple abelian variety.  Combining Lemma~\ref{claim:red 1 abelian} and Corollary~\ref{claim:red 2 abelian}, along with the proof of Theorem~\ref{thm:endomorphisms} for powers of simple abelian varieties yields the conclusion in Theorem~\ref{thm:endomorphisms} for any abelian variety. 
\end{proof}

Combining the fact that we have proven the conclusion of Theorem~\ref{thm:endomorphisms} for abelian varieties and also for tori, we can complete the proof of the general case in Theorem~\ref{thm:endomorphisms}.

\begin{proof}[Proof of Theorem~\ref{thm:endomorphisms}.]
First we note that we proved that when $G$ is either an abelian variety or an algebraic torus, then there is no proper subvariety intersecting the orbit of each torsion point of $G$.  For the general case of a semiabelian variety $G$, we know there exists a short exact sequence of algebraic groups:
\begin{equation}
\label{eq:short exact sequence}
1\lra \bG_m^N\lra G\lra A\lra 1,
\end{equation}
where $A$ is an abelian variety and $\bG_m^N$ is the toric part of $G$. Furthermore, $\Phi$ induces a group endomorphism $\bar{\Phi}$ of $A$ and also, we have that $\Phi|_{\bG_m^N}$ is a group endomorphism of $\bG_m^N$ (since there exist no non-trivial morphisms between an abelian variety and an algebraic torus; see also \cite[Fact~2.3]{GS-semiabelian}). Then Lemma~\ref{claim:red 3 semiabelian} yields the desired conclusion.
\end{proof}

Theorem~\ref{thm:endomorphisms} is an important ingredient in proving Theorem~\ref{thm:translations}.

\begin{proof}[Proof of Theorem~\ref{thm:translations}.]
We first note that each regular self-map on a semiabelian variety is a composition of a translation with an algebraic group endomorphism (see \cite[Fact~2.1]{GS-semiabelian}).

We start by proving some useful reductions, similar to the ones established in the beginning of our proof of Theorem~\ref{thm:endomorphisms}.

\begin{claim}
\label{claim:red 1 semiabelian}
It suffices to prove Theorem~\ref{thm:translations} for an isogenous dynamical system, i.e., if there exists an isogeny $\tau:G\lra G'$ of semiabelian varieties endowed with dominant, regular self-maps $\Phi:G\lra G$ and $\Phi ':G'\lra G'$ such that $\Phi'\circ \tau = \tau\circ \Phi$, then the conclusion of Theorem~\ref{thm:translations} for $(G,\Phi)$ yields the conclusion of Theorem~\ref{thm:translations} for $(G',\Phi')$.
\end{claim}

\begin{proof}[Proof of Claim~\ref{claim:red 1 semiabelian}.]
First we show that knowing conclusion~(A) in Theorem~\ref{thm:translations} holds for the proper subvariety $X\subset G$ (under the action of $\Phi$) yields that also conclusion~(A) holds for the proper subvariety $X':=\tau(X)\subset G'$ (under the action of $\Phi'$). Indeed, given any proper, irreducible, periodic subvariety $Z'\subset G'$ (under the action of $\Phi'$), we let $Z\subset G$ be an irreducible component of $\tau^{-1}(Z')$, which is periodic under the action of $\Phi$. The fact that conclusion~(A) in Theorem~\ref{thm:translations} holds for the dynamical system $(G,\Phi)$ yields that $Z\subseteq X$; so, 
$Z'=\tau(Z)\subseteq \tau(X)=X'$, 
as claimed. 

Next we show that if conclusion~(B) holds for $(G,\Phi)$, then it must also hold for $(G',\Phi')$. Indeed, assume on the contrary that there exists some proper subvariety $Y'\subset G'$ which contains an iterate of each periodic, irreducible, proper subvariety of $G'$ (under the action of $\Phi'$). We claim that $Y:=\tau^{-1}(Y')$, which is a proper subvariety of $G$, must contain an iterate of each irreducible, proper, periodic subvariety of $G$ (under the action of $\Phi$), therefore contradicting conclusion~(B) for the dynamical system $(G,\Phi)$. Let $Z\subset G$ be a proper, irreducible, periodic subvariety of $G$ and let $Z':=\tau(Z)\subset G'$; clearly, $Z'$ is a proper, irreducible, periodic subvariety of $G'$. Then there exists some $m\in\N$ such that $(\Phi')^m(Z')\subseteq Y'$; so,
$$(\Phi')^m(\tau(Z))=\tau(\Phi^m(Z))\subseteq Y',$$
which yields that $\Phi^m(Z)\subseteq \tau^{-1}(Y')=Y$, as claimed.

This finishes our proof of Claim~\ref{claim:red 1 semiabelian}.
\end{proof}

\begin{claim}
\label{claim:red 2 semiabelian}
With the notation as in Theorem~\ref{thm:translations}, let $\tau:G\lra G$ be an automorphism and let $\Phi ' :=\tau\circ \Phi\circ \tau^{-1}$. Then conclusion~(A) holds for $(G,\Phi)$ if and only if it holds for $(G,\Phi')$ and also, conclusion~(B) holds for $(G,\Phi)$ if and only it holds for $(G,\Phi')$.
\end{claim}

\begin{proof}[Proof of Claim~\ref{claim:red 2 semiabelian}.]
This is an immediate consequence of Claim~\ref{claim:red 1 semiabelian}.
\end{proof}

\begin{claim}
\label{claim:product}
Assume conclusion~(B) from Theorem~\ref{thm:translations} holds for the dynamical system $(G_1,\Phi_1)$, where $G_1$ is a semiabelian varieties endowed with a dominant, regular self-map $\Phi_1$. Then for any semiabelian variety $G_2$ endowed with a dominant, regular self-map $\Phi_2$, we have that also conclusion~(B) from Theorem~\ref{thm:translations} holds for  the semiabelian variety $G_1\times G_2$ endowed with the dominant, regular self-map $\Phi$ given by the rule $(x_1,x_2)\mapsto (\Phi_1(x_1), \Phi_2(x_2))$.
\end{claim}

\begin{proof}[Proof of Claim~\ref{claim:product}.]
Let $Y\subset G_1\times G_2$ be a subvariety containing some iterate of each proper, irreducible, periodic subvariety of $G_1\times G_2$ (under the action of $\Phi$). We know that there is no proper subvariety of $G_1$ which contains some iterate of each proper, irreducible, periodic subvariety of $G_1$, under the action of $\Phi_1$. Then for any proper, irreducible, periodic subvariety $V_1\subset G_1$ (under the action of $\Phi_1$), we have that $V_1\times G_2$ is a proper, irreducible, periodic subvariety of $G$ (under the action of $\Phi$); thus $Y$ must contain a variety of the form $\Phi_1^i(V_1)\times G_2$ for some positive integer $i$. Our hypothesis regarding the dynamical system $(G_1,\Phi_1)$ yields that $Y$ must project dominantly onto $G_1$, and moreover  $Y=G_1\times G_2$, as claimed.
\end{proof}

We know that $\Phi$ is the composition of a translation (by some point $y\in G(K)$) and a dominant group endomorphism $\Phi_0:G\lra G$. The next reduction is crucial for our proof.

\begin{claim}
\label{claim:nilpotent}
It suffices to prove the conclusion under the assumption that $\Phi_0-\id$ is a nilpotent endomorphism of $G$ (where $\id$ will always represent the identity morphism).
\end{claim}

\begin{proof}[Proof of Claim~\ref{claim:nilpotent}.]
Let $f\in \Z[t]$ be the minimal, monic polynomial satisfied by $\Phi_0$, i.e., $f(\Phi_0)=0$. Since $\Phi$ (and therefore also $\Phi_0$) is dominant, the roots of $f(t)$ are nonzero. We let $r$ be the order of $1$ as a root of $f(t)$ and also, we let $g\in \Z[t]$ such that $f(t)=(t-1)^r\cdot g(t)$. Then we let $G_1:=(\Phi_0-\id)^r(G)$ and $G_2:=(g(\Phi_0))(G)$. 

Because the polynomials $(t-1)^r$ and $g(t)$ are coprime, coupled with the fact that the group of $K$-rational points of a semiabelian variety is divisible, we get that $G_1+G_2=G$; also, we have that $G_1\cap G_2$ is finite (these facts are proven in \cite[Lemma~6.1]{GS-abelian} and \cite[(4.0.2)]{GS-semiabelian}). Thus $G$ is isogenous with the direct product $G_1\times G_2$; more precisely, we consider the isogeny $\tau:G_1\times G_2\lra G$ given by $\tau(x,y)=x+y$. 

We let $y_1\in G_1(K)$ and $y_2\in G_2(K)$ such that $y=y_1+y_2$; then we let $\Phi_1:G_1\lra G_1$ be the composition of $(\Phi_0)|_{G_1}$ with the translation by $y_1$, while $\Phi_2:G_2\lra G_2$ is given by the composition of $(\Phi_0)|_{G_2}$ with the translation by $y_2$. We get that for each $x_1\in G_1$ and $x_2\in G_2$ we have:
$$\Phi(\tau(x_1,x_2))=\tau(\Phi_1(x_1),\Phi_2(x_2));$$
in other words, the dynamical system $(G,\Phi)$ is isogenous to the dynamical system $(G_1\times G_2,\Phi_1\times \Phi_2)$. Thus, by Claim~\ref{claim:red 1 semiabelian}, it suffices to prove the conclusion of Theorem~\ref{thm:translations} for $(G_1\times G_2,\Phi_1\times \Phi_2)$. 

We immediately observe that conclusion~(B) from Theorem~\ref{thm:translations} always holds for $(G_1,\Phi_1)$, since in this case, our problem reduces to Theorem~\ref{thm:endomorphisms}. Indeed, we observe that $\Phi_0-\id$ is dominant on $G_1$ (note that the minimal polynomial satisfied by $\Phi_0$ as an endomorphism of $G_1$ is $g(t)$, which has no root equal to $1$). Thus there exists some $y_3\in G_1(K)$ such that $\Phi_0(y_3)-y_3=-y_1$ (note that each dominant group endomorphism is actually surjective) and this yields  (after denoting by $\tau_z$ the translation-by-$z$ map on $G_1$ for any point $z\in G_1(K)$) that 
$$(\tau_{-y_3}\circ \Phi_1\circ \tau_{y_3})(x)=\Phi_0(x)\text{ for each }x\in G_1;$$ 
therefore $\tau_{-y_3}\circ\Phi_1\circ \tau_{y_3}$ is a group endomorphism of $G_1$. Using Claim~\ref{claim:red 2 semiabelian}, we see that it is no loss of generality to assume that $y_1=0$, i.e., $\Phi_1:G_1\lra G_1$ is a dominant group endomorphism, and then, as proven in Theorem~\ref{thm:endomorphisms}, we know that alternative~(B) holds in this case. 

Using Claim~\ref{claim:product} (and also Claim~\ref{claim:red 1 semiabelian}), we are left to proving Theorem~\ref{thm:translations} for the dynamical system $(G_2,\Phi_2)$, in which case we do know (as claimed) that $\Phi_2-\id$ is a nilpotent endomorphism of $G_2$. This finishes our proof of Claim~\ref{claim:nilpotent}.
\end{proof}

From now on, we work with the extra hypothesis on the dynamical system $(G,\Phi)$ established by Claim~\ref{claim:nilpotent} and thus assume that $\Phi=\tau_y\circ\Phi_0$ (where $\tau_y$ is the translation-by-$y$ map on $G$)  and furthermore $\Psi:=\Phi_0-\id$ is a nilpotent endomorphism of $G$. We prove that  Theorem~\ref{thm:translations} must hold under these assumptions.

Let $W$ be an irreducible, proper subvariety of $G$ that is of maximal dimension among all irreducible, periodic subvarieties under the action of $\Phi$; clearly, if there is no such $W$, i.e., there is no proper subvariety of $G$, periodic under the action of $\Phi$, then alternative~(A) holds trivially.

We note that the conclusion in Theorem~\ref{thm:translations} is preserved if we replace $\Phi$ by an iterate $\Phi^N$; the argument is identical with the proof of Lemma~\ref{lem:iterate}.

With the above observation, we may assume that $W$ is fixed by $\Phi$.
\begin{claim}
\label{claim:1-translate}
For each $x\in \ker(\Psi)$, we have that $x+W$ is also fixed by $\Phi$.
\end{claim}  

\begin{proof}[Proof of Claim~\ref{claim:1-translate}.]
Let $w\in W$; we know that $\Phi(w)\in W$ and furthermore, since $x\in \ker(\Phi_0-\id)$, we have that $\Phi_0(x)+\Phi(w)\in x+W$. But  $\Phi_0(x)+\Phi(w)=\Phi(x+w)$ and therefore, $\Phi(x+w)\in x+W$ for each $w\in W$, thus proving that $x+W$ is indeed fixed by $\Phi$.
\end{proof}

Claim~\ref{claim:1-translate} yields that the Zariski closure $Z$ of $\ker(\Psi)+W$ is fixed by $\Phi$. Since $W$ has maximal dimension among the periodic, proper, subvarieties of $G$, we must have that either $Z=G$, or that $W$ is an irreducible component of $Z$ (and that they both have the same dimension). We split our analysis into these two cases.

{\bf Case 1.} Assume $Z=G$.

In this case, we get that there is no proper subvariety $Y\subset G$ which contains an iterate of each periodic, proper, irreducible subvariety of $G$, because each $x+W$ (for $x\in \ker(\Psi)$) is fixed by $\Phi$ and their union is Zariski dense in $G$, according to our assumption.

{\bf Case 2.} $W$ is an irreducible component of $Z$ (of the same dimension).

In this case, we let $W_1:=Z$, which is a proper subvariety of $G$, fixed under the action of $\Phi$; as observed, $W_1$ may no longer be irreducible.
\begin{claim}
\label{claim:2-translation}
For each $x\in \ker(\Psi^2)$, we have that $x+W_1$ is also fixed by $\Phi$.
\end{claim}

\begin{proof}[Proof of Claim~\ref{claim:2-translation}.]
Indeed, for each $w\in W_1$ and each $x\in\ker(\Psi^2)$, we have that 
$$\Phi(x+w)=\Phi(w)+\Phi_0(x)=x+\Psi(x)+\Phi(w)\in x+W_1,$$
because $\Phi(w)\in W_1$ (since $W$ and therefore, also $W_1$ is fixed by $\Phi$), while $\Psi(x)\in\ker(\Psi)$ and $W_1=\ker(\Psi)+W_1$ since $W_1=Z$ is the Zariski closure of $W+\ker(\Psi)$. Hence $x+W_1$ is fixed by $\Phi$, as claimed.   
\end{proof}

Claim~\ref{claim:2-translation} yields that for each $x\in \ker(\Psi^2)$, the subvariety $x+W$ is fixed by $\Phi^{N_1}$, where $N_1$ is a positive integer depending solely on the number of irreducible components of $W_1=Z$. Without loss of generality, we replace again $\Phi$ by its iterate $\Phi^{N_1}$ such that each $x+W$ is now fixed by $\Phi$ (for each $x\in\ker(\Psi^2)$).

We argue as before and if $\ker(\Psi^2)+W_1=G$, we have that conclusion~(B) holds in Theorem~\ref{thm:translations} because the union of all irreducible,  proper subvarieties  of $G$, which are fixed by $\Phi$, is actually Zariski dense in $G$. Hence it remains to analyze the case $W_2:=\ker(\Psi^2)+W_1=\ker(\Psi^2)+W$ is a subvariety of $G$ of the same dimension as $W$. Then we argue as in the proofs of Claims~\ref{claim:1-translate}~and~\ref{claim:2-translation} and inductively derive that unless conclusion~(B) holds in Theorem~\ref{thm:translations}, then we must have that $\ker(\Psi^j)+W$ is a proper subvariety of $G$ for any $j\in\N$. However, $\Psi$ is nilpotent and so we conclude (in finitely many steps) that conclusion~(B) must hold because $\ker(\Psi^r)=G$ for sufficiently large integers $r$.

This finishes our proof of Theorem~\ref{thm:translations}.
\end{proof}


\section{Proof of Theorem~\ref{thm:primitive}}
\label{sec:proofs 2}

We first recall some basic facts about skew polynomial rings.  Given a ring $R$ with an automorphism $\sigma$, an ideal $I$ is \emph{$\sigma$-invariant} if $\sigma(I)\subseteq I$; a $\sigma$-invariant ideal $I$ is \emph{$\sigma$-prime}  if whenever $JL\subseteq I$ with $J$, $L$ $\sigma$-invariant ideals, then we must  have $J\subseteq I$ or $L\subseteq I$.  Then given a prime ideal $P$ of either $R[t;\sigma]$ or $R[t^{\pm 1};\sigma]$, where in the former case we assume $t\not\in P$, we obtain a $\sigma$-prime ideal $P\cap R$ of $R$ and conversely if $I$ is a $\sigma$-prime ideal of $R$ then $IR[t;\sigma]$ is a prime ideal of $R[t;\sigma]$ that does not contain $t$ and similarly $IR[t^{\pm 1};\sigma]$ is a prime ideal of $R[t^{\pm 1};\sigma]$. In general this map from $\sigma$-prime ideals of $R$ to prime ideals of $R[t;\sigma]$ is injective and not surjective.  We note that a $\sigma$-prime ideal is a \emph{semiprime} ideal of $R$ and not necessarily prime; that is, it is an intersection of prime ideals of $R$. 

For Theorem \ref{thm:primitive} we are interested in three different properties of prime ideals: being primitive, being rational, and being locally closed in the prime spectrum.  We recall that given a field $k$, a $k$-algebra $R$ satisfies the strong Nullstellensatz if every prime ideal is an intersection of primitive ideals and if whenever $M$ is a simple $R$-module, the endomorphism ring ${\rm End}_R(M)$ is a finite-dimensional $k$-algebra.  If $R$ is a finitely generated $k$-algebra then $R[t;\sigma]$ and $R[t^{\pm 1};\sigma]$ satisfy the strong Nullstellensatz \cite[Theorem 4.6]{McC82}.  For an algebra $S$ satisfying the strong Nullstellensatz we always have the following implications for $P\in {\rm Spec}(S)$:
$$P~{\rm locally ~closed}\implies P~{\rm primitive}\implies P~{\rm rational}.$$

In order to prove Theorem~\ref{thm:primitive}, we need a result about invariant subvarieties of $\mathbb{G}_m^d$; it would be interesting (by itself) to know whether an analogue of this result holds in a more general setting. Before stating our lemma, we recall that an endomorphism $\Phi$ of some  quasiprojective variety $X$, which is not necessarily irreducible, preserves a non-constant fibration if there exists a rational function $f:X\dashrightarrow \bP^1$ such that $f\circ \Phi=f$ and such that $f$ is defined and non-constant on a dense open subset of some irreducible component of $X$; this notion plays a crucial role in the Medvedev-Scanlon conjecture from \cite{MS-Annals} regarding Zariski dense orbits.

\begin{lemma} Let $n\in\mathbb{N}$ and let $\phi: \mathbb{G}_m^n\to \mathbb{G}_m^n$ be the composition of a translation with an algebraic group  automorphism.  If $X\subseteq \mathbb{G}_m^n$ is an  irreducible subvariety defined over an algebraically closed field $k$ of characteristic $0$ that is invariant under $\phi$ and moreover, if $\phi|_X$ does not preserve a non-constant fibration, then $X$ is a translate of some subtorus of $\bG_m^n$.
\label{lem:red}
\end{lemma}

\begin{proof}
Let $Y\subseteq \mathbb{G}_m^n$ be an irreducible subvariety of minimal dimension with respect to having the following properties:
\begin{enumerate}
\item[(i)] $Y$ is invariant under some power of $\phi$;
\item[(ii)] $Y$ is a translate of a subtorus of $\bG_m^n$;
\item[(iii)] $Y\supseteq X$.
\end{enumerate}
Such a $Y$ exists by hypothesis.  We claim that $Y=X$.  To see this, we have
$\mathcal{O}(Y) = k[x_1^{\pm 1},\ldots ,x_d^{\pm 1}]$, where $d=\dim(Y)$.  If $\mathcal{O}(Y)\neq \mathcal{O}(X)$ then there is a non-trivial linear combination 
$$\sum_{i=1}^r c_i m_i$$ that vanishes identically on $X$, where $c_1,\ldots ,c_r\in k^*$ and $m_1,\ldots ,m_r$ are distinct monomials in $x_1^{\pm 1},\ldots ,x_d^{\pm 1}$.  We may assume that $r$ is minimal with respect to such a dependence holding on $X$ and since the $m_i$ are units in $\mathcal{O}(X)$, we see that $r\ge 2$.

Then let $H$ denote the subgroup of $\mathcal{O}(X)$ generated by the images of $x_1^{\pm 1},\ldots ,x_d^{\pm 1}$.  By replacing $\phi$ by a power, we may assume that $X$ and $Y$ are both $\phi$-invariant.
For every $t\ge 1$ we have
$$\sum_{i=1}^r c_i (\phi^*)^t(m_i)$$ vanishes on $X$, and we have $(\phi^*)^t(m_i)\in k^* H$ for every $t$.
In particular, by the theory of $S$-unit equations for function fields (see, for example, \cite{Zannier}) we have that 
there exist $s<t$ such that
$$(\phi^*)^t(m_i)=\lambda_i (\phi^*)^s(m_i)$$ on $X$ for $i=1,\ldots ,r$ for some $\lambda_1,\ldots ,\lambda_r\in k^*$.
Since $\phi$ is an automorphism, this then gives us that $(\phi^*)^{t-s}(m_i)=\lambda_i m_i$ on $X$ for $i=1,\ldots, r$.  In particular the fact that
$$\sum_{i=1}^r c_i m_i$$ vanishes identically on $X$ gives that
$$\sum_{i=1}^r c_i (\phi^*)^{t-s}(m_i)$$ vanishes identically on $X$ and so
$\sum_{i=1}^r c_i \lambda_i m_i=0$ on $X$.  In particular, we have
$$\sum_{i=2}^r c_i (\lambda_i-\lambda_1) m_i=0$$ on $X$.  By minimality of $r$ in our relation, we see that $\lambda_1=\cdots =\lambda_r$.  Then since $r\ge 2$ we have
$$(\phi^*)^{t-s}(m_1/m_2) = m_1/m_2.$$  Letting $h=m_1/m_2$ we then see that either $h$ is constant on $X$ or $\phi^{t-s}|_X$ preserves a non-constant fibration.  But the latter possibility would then give that $\phi|_X$ preserves a non-constant fibration (see \cite[Lemma~2.1]{3-folds}), which is a contradiction.  It follows that there is some $\lambda\in k^*$ such that $h=\lambda$ on $X$.  Then we can write $h=x_1^{i_1}\cdots x_d^{i_d}$ with $i_1,\ldots ,i_d\in \mathbb{Z}$, not all zero since $m_1\ne m_2$.  We may also assume that the integers $i_1,\ldots ,i_d$ are coprime since otherwise we could replace $h$ by a root of it in $H$ and it would still be identically equal to some nonzero value of $k$ on $X$ since $k$ is algebraically closed. 
Then we have a surjective map
$$\mathbb{Z}^d\to \mathbb{Z}$$ given by $(a_1,\ldots ,a_d)\mapsto \sum a_j i_j$ and since $\mathbb{Z}$ is projective as a module over itself the sequence splits and since all projectives are free, we see that
$\mathcal{O}(Y)^*/k^* \cong  \langle h\rangle \oplus \langle h_1,\ldots ,h_{d-1}\rangle,$ where $h_1,\ldots ,h_{d-1}$ are monomials in $x_1^{\pm 1},\ldots ,x_d^{\pm 1}$.  In particular,
$\mathcal{O}(Y)/(h-\lambda)\cong \mathcal{O}(\mathbb{G}_m^{d-1})$.   Let $Y'=Y\cap V(h-\lambda)$.  Then $Y'$ contains $X$, it is $\phi^{t-s}$-invariant and irreducible, and also it is a translate of an algebraic subtorus, contradicting the minimality of $Y$.  It follows that $Y=X$ and the result follows.
 \end{proof}
 We note that a subvariety not preserving a non-constant fibration is related to the property of being rational---this connection is made precise in Lemma \ref{lem:rat}.
Finally, we need a result relating being locally closed to $\sigma$-invariant ideals.
 \begin{lemma}
 Let $k$ be a field and let $R$ be a finitely generated commutative $k$-algebra that is an integral domain and let $\sigma$ be a $k$-algebra automorphism of $R$.  Then if $R$ is not a field then $(0)$ is a locally closed prime ideal of $R[x;\sigma]$ if and only if the intersection of the nonzero $\sigma$-prime ideals of $R$ is nonzero. 
 \label{lem:lc}
 \end{lemma}
 \begin{proof} If the intersection of the nonzero $\sigma$-prime ideals of $R$ is zero, then the intersection of the nonzero prime ideals of $R[x;\sigma]$ is zero and so $(0)$ is not locally closed.  It suffices to show that if the intersection of the nonzero $\sigma$-prime ideals of $R$ is nonzero then $(0)$ is locally closed.  Henceforth we assume that the intersection of the nonzero $\sigma$-prime ideals of $R$ is nonzero.  
There are three types of nonzero prime ideals in ${\rm Spec}(R[t;\sigma])$:
\begin{enumerate}
\item[(i)] primes $Q$ such that $Q\cap R\neq (0)$ and $t\not \in Q$;
\item[(ii)] primes $Q$ such that $t\in Q$;
\item[(iii)] primes $Q$ such that $t\not\in Q$ and $Q\cap R=(0)$.
\end{enumerate}
Since a finite intersection of nonzero ideals in a prime ring is nonzero, it suffices to show that the intersection of the nonzero primes of each type is nonzero.  Note that if $t\not\in Q$ then $Q\cap R$ is $\sigma$-invariant, since for $q\in Q$ we have $t q =\sigma(q) t\in Q$ and since $tS[t;\tau]=S[t;\tau]t$ we then have $$\sigma(q) S[t;\tau] t\subseteq Q,$$ and so $\sigma(q)\in Q$ since $Q$ is prime and $t\not\in Q$.  Then the intersection of prime ideals of type (i) is nonzero since $Q\cap R$ is a $\sigma$-prime ideal when $t\not\in Q$ and we are assuming that the intersection of nonzero $\sigma$-prime ideals is nonzero.  

For primes of type (ii), we have $t$ is in the intersection.  

We note that if $t\not \in Q$ then no power of $t$ is in $Q$, because $tR[t;\sigma]=R[t;\sigma]t$ and so if $t^m\in Q$ then $(t)^m\subseteq Q$ and since $Q$ is prime we then have $(t)\subseteq Q$, which we have assumed is not the case.  Thus primes of type (iii) cannot contain an element of the form $s t^m$ with $m\ge 1$ and $s\in R\setminus \{0\}$, and hence they survive in the localization 
$K[t^{\pm 1};\sigma]$, where $K$ is the field of fractions of $R$.  But this algebra is simple unless $\sigma$ has finite order \cite[\S6]{Jor93} and so we see that nonzero primes of type (iii) do not exist unless $\sigma$ has finite order.  But in this case $R$ satisfies a polynomial identity and so $(0)$ is not locally closed in $R[t;\sigma]$ and the intersection of the nonzero $\sigma$-prime ideals of $R$ is zero.  This can be seen by noting that $R$ is a finitely generated integral domain that is not a field and hence $(0)$ is an intersection of the maximal ideals and since $\sigma$ has finite order we see that the finite intersection of the orbit of each maximal ideal is a nonzero $\sigma$-prime ideal and the intersection of all such ideals is equal to zero.
The result follows.   
 \end{proof}

We are now ready to prove Theorem~\ref{thm:primitive} (a).

\begin{proof}[Proof of Theorem~\ref{thm:primitive} (a).]  
By the Irving-Small reduction techniques for relating locally closed and primitive ideals (see \cite[Lemma 8.4.28]{Row2}) we may assume that $k$ is algebraically closed.  Let $S=R[t;\sigma]$.  We divide the proof into three cases.  The first case is when $t\not\in P$ and $(P\cap R)S$ is strictly contained in $S$.  Let $I=P\cap R$ and observe that $I$ is $\sigma$-prime and that $\sigma$ induces a non-trivial automorphism of $R/I$.  Then $S/P$ is a non-trivial homomorphic image of $(R/I)[t;\sigma]$.  Let $\bar{P}$ denote the image of $P$ in $(R/I)[t;\sigma]$.  Then
$\bar{P}\cap (R/I)=(0)$ and so we see that $\bar{P}$ survives in the localization 
$Q(R/I)[t;\sigma]$ and we let $\tilde{P}$ denote the ideal generated by $\bar{P}$ in this localization.  Now $\tilde{P}$ necessarily contains a monic polynomial, since the collection of leading coefficients of elements of $\tilde{P}$ forms a nonzero $\sigma$-invariant ideal $L$ of $Q(R/I)$ and since $I$ is a $\sigma$-prime ideal of $R$, we have that $L$ is necessarily all of $Q(R/I)$.  In particular,
$Q(R/I)[t;\sigma]/\tilde{P}$ is a finite $Q(R/I)$-module and thus it satisfies a polynomial identity.  It follows that $(R/I)[t;\sigma]/\bar{P}$ satisfies a polynomial identity.  We note that $P$ is primitive if and only if $\bar{P}$ is primitive and since in a polynomial identity algebra primitive ideals are precisely those ideals that are maximal we see that $P$ is primitive if and only if it is maximal in this case.  In particular, $P$ is vacuously locally closed.  This completes the first case.

The second case is when $t\in P$.  In this case $S/P$ is a homomorphic image of $R$ and so $P$ is primitive if and only if it is maximal, and this holds if and only if $P$ is locally closed.

Finally, we may assume that $I=P\cap R$ has the property that $IS=P$ and $t\not\in P$.  Since $tS=St$ and $P$ is a prime ideal, we have that no power of $t$ is in $P$.  Let $G={\rm Spec}(R)$.  Let $X$ be the zero set of $I$.  Then $X$ is $\sigma$-invariant and $\sigma$ acts transitively on the irreducible components of $X$.  We let $X_1,\ldots ,X_d$ denote the irreducible components of $X$ and let $J=I(X_1)$.  Then $J$ is prime and is $\sigma^m$-invariant for some $m$.  By the Leroy-Matzcuk theorem, we have that $(R/P)[t;\sigma]$ is primitive if and only if $(R/J)[t^m;\sigma^m]$ is primitive (see \cite[Corollary 2.2]{LM96}) and by \cite{Theorem 2.3}[Let89], to show that $(0)$ is locally closed in $(R/P)[t;\sigma]$, it suffices to show that $(0)$ is locally closed in $(R/J)[t^m;\sigma^m]$.  Suppose first that $(0)$ is primitive.  Then since $(R/J)[t^m;\sigma^m]$ satisfies the strong Nullstellensatz, we have that $(0)$ is rational.  Since $(0)$ is rational, we have that $X_1={\rm Spec}(R/J)\cong \mathbb{G}_m^d$ for some $d\ge 0$ by Lemma \ref{lem:rat} and Lemma \ref{lem:red}.  By the Leroy-Matczuk theorem, we have that $(0)$ is primitive if and only if $\tau:=\sigma^m$ has infinite order of $X_1$ and there is some hypersurface $Y\subseteq X_1$ such that every proper $\tau$-periodic subvariety of $X_1$ has the property that some irreducible component is contained in $Y$.  In particular, by Theorem \ref{thm:translations} we have that this occurs if and only if the union of the proper $\tau$-periodic subvarieties of $X_1$ is not Zariski dense, which is the same as saying that the intersection of the nonzero $\tau$-prime ideals or $R/J$ is nonzero.  Then by Lemma \ref{lem:lc}, we see that this gives that $(0)$ is locally closed, unless $X_1$ is a point.  But if $X_1$ is a point, we have $(R/J)[t^m;\sigma^m]$ is isomorphic to $k[x]$ and $(0)$ is neither primitive nor locally closed in this case. 

Conversely, if $(0)$ is locally closed then $(0)$ is primitive since $R[t;\sigma]$ satisfies the strong Nullstellensatz \cite[Proposition 8.4.18]{Row2}.\end{proof}

In the case where we have a skew Laurent polynomial ring $S=R[x^{\pm 1};\sigma]$ the criterion for primitivity is due to Jordan \cite[Theorem 7.3]{Jor93}: in this case, $(0)$ is a primitive ideal if and only if either $R$ has a $\sigma$-special element or there is a maximal ideal $Q$ with the property that $\bigcap_{n\in \mathbb{Z}} \sigma^n(Q)=(0)$.  The property of the existence of $Q$ such that $\bigcap_{n\in \mathbb{Z}} \sigma^n(Q)=(0)$ is, conjecturally, equivalent to $\sigma$ not preserving a non-constant fibration.  (This in fact would be implied by the Medvedev-Scanlon conjecture \cite{MS-Annals}.)  Furthermore, the following easy lemma provides an interesting reduction in Jordan's criterion.

\begin{lemma}
\label{lem:red Jor}
Let $X$ be a quasiprojective variety defined over an algebraically closed field $k$ of characteristic $0$, endowed with an automorphism $\sigma$. Assume there exists a proper subvariety $Y\subset X$ which contains an iterate of each proper, irreducible, periodic subvariety of $X$, under the action of $\sigma$. Then there exists a point $x\in X(k)$ whose orbit under $\sigma$ is Zariski dense in $X$. 
\end{lemma}

\begin{proof}
As proven by Amerik \cite{Amerik}, there exists a point $x\in X(k)$ whose entire orbit under $\sigma$ avoids the proper, Zariski-closed subset $Y\subset X$. We claim that the orbit of $x$ under $\sigma$ must be Zariski dense in $X$. Indeed, otherwise, it must be a finite union of proper, irreducible, periodic subvarieties of $X$; then let $Z$ be any one of these subvarieties. By hypothesis, there exists $i\ge 0$ such that $\sigma^i(Z)\subset Y$. On the other hand, by construction, there exists some $m\ge 0$ such that $\sigma^m(x)\in Z$. Thus $\sigma^{m+i}(x)\in Y$, which contradicts our choice for $x$.

This concludes our proof of Lemma~\ref{lem:red Jor}.
\end{proof}

Lemma~\ref{lem:red Jor} implies that if there is a $\sigma$-special element then there is a maximal ideal $Q$ with the property that $\bigcap_{n\in \mathbb{Z}} \sigma^n(Q)=(0)$ and so we can simplify Jordan's criterion for finitely generated noetherian algebras $R$.  We record this in the following statement.

\begin{prop} \label{rem:Jordan} Let $k$ be an algebraically closed field and let $R$ be a finitely generated integral domain with an automorphism $\sigma$.  Then $R[t^{\pm 1};\sigma]$ is primitive if and only if there is a closed $k$-point $x\in X:={\rm Spec}(R)$ whose orbit under the map induced by $\sigma$ is Zariski dense in $X$.
\end{prop}
\begin{proof}
This follows immediately from Jordan's theorem \cite[Theorem 7.3]{Jor93} along with Lemma~\ref{lem:red Jor}.
\end{proof}

 \begin{lemma}
 Let $k$ be an algebraically closed field and let $R$ be a finitely generated reduced commutative $k$-algebra and let $\sigma$ be a $k$-algebra automorphism of $R$.  Then $(0)$ is a rational ideal of $S=R[x^{\pm 1};\sigma]$ if and only if either $P$ is maximal or $P\cap R$ has infinite codimension in $R$ and $P=(R\cap P)S$ and the automorphism of $X={\rm Spec}(R/(P\cap R))$ induced by $\sigma$ does not preserve a non-constant fibration. \label{lem:rat}
 \end{lemma}
 \begin{proof} Let $Q=P\cap R$.  Then $Q$ is a semiprime $\sigma$-invariant ideal of $R$.  Then we may replace $R$ by $R/Q$ and we may assume that $P\cap R=(0)$.  If $P\neq (0)$ then by a result of Irving \cite[Theorem 4.3]{Irv79} we have that $\sigma$ has finite order on $R$ and so $R[x^{\pm 1};\sigma]$ satisfies a polynomial identity.  Since the only rational ideals of polynomial identity algebras are maximal ideals, by Kaplansky's theorem \cite[Theorem 6.1.25]{Row2}, we see that $P$ is maximal in this case.
  
 So now we may assume that $P=(0)$; i.e., $QS=P$.  Then $Q$ is semiprime and we let $Q_1,\ldots ,Q_s$ denote the prime ideals of $R$ that are minimal above $Q$.  Then $\sigma$ permutes $Q_1,\ldots ,Q_s$ and the group generated by $\sigma$ acts transitively on this set. Let $X={\rm Spec}(R/Q)$. Then $X$ is the union of the irreducible subvarieties $X_1,\ldots ,X_s$, where $X_i={\rm Spec}(R/Q_i)$.  Moreover, there is some $\tau=\sigma^m$ such that $Q_1,\ldots, Q_s$ are $\tau$-stable.  Furthermore, $(0)$ is a rational prime ideal of $(R/Q)[t^{\pm 1};\sigma]$ if and only if $(0)$ is a rational prime ideal of $(R/Q_1)[x^{\pm m};\tau]$ (see \cite[Corollary 1.2]{Let89}) and the automorphism of $X$ induced by $\sigma$ preserves a non-constant fibration if and only if the automorphism of $X_1$ induced by $\tau$ preserves a non-constant fibration (see \cite[Lemma~2.1]{3-folds}).  Then if the automorphism induced by $\tau$ does not preserve a non-constant fibration, there is a non-constant rational map $f:X_1\to \mathbb{P}^1$ such that the automorphism of $X_1$ induced by $\tau$ preserves the fibres of $f$.  In particular if we take the embedding $f^*: k(\mathbb{P}^1)\to k(X_1)$, then by construction we have $\tau$ is the identity on the image of $f^*$ and so  preserves the function field and so $f^*(k(\mathbb{P}^1))$ is a central subfield of the division ring of quotients of 
 $(R/Q_1)[t^{\pm m};\tau]$ and so $(0)$ is not a rational prime of $(R/Q_1)[t^{\pm m};\tau]$; therefore, $P$ is not rational.
 
  Next observe that if $P$ is not rational then $P_1:=Q_1R[t^{\pm m};\tau]$ is not rational. In particular, we have the Goldie ring of quotients of $(R/Q_1)[t^{\pm m};\tau]$ has a central element that is not in $k$.  Let $K$ denote the field of fractions of $R/Q_1$.  Then since $K((t^m;\tau))$ is a division ring, we have $Q((R/Q_1)[t^{\pm m};\tau])$ embeds in this skew power series ring and thus if we let $z$ denote an element of $Q( (R/Q_1)[t^{\pm m};\tau])$ that is not algebraic over $k$ then we have a Laurent power series expansion
 $$z=\sum_{i\ge -M} \alpha_i t^{mi},$$ with $\alpha_i\in K$ and $\alpha_{-M}\neq 0$.  Now $z$ commutes with $t^m$ and hence $$ \sum_{i\ge -M} \tau(\alpha_i)t^{mi+m} = t^m z = z t^m =  \sum \alpha_i t^{mi+m},$$ and so
 $\tau(\alpha_i)=\alpha_i$ for all $i\ge -M$.  In particular, if some $\alpha_i\in K\setminus k$ then $\alpha_i$ is a rational function on ${\rm Spec}(R/Q_i)$ and $\tau(\alpha_i)=\alpha_i$ and thus the automorphism induced by $\tau$ preserves a non-constant fibration.  If, on the other hand, $\alpha_i\in k$ for all $i$ then the fact that $z$ is central and using the fact that $az=za$ for all $a\in R/Q_1$ shows that $\tau^j(a)=a$ for all $a\in R/Q_1$ for which $\alpha_j\neq 0$.  Thus either $\tau$ has finite order on $R/Q_1$ or $\alpha_j=0$ for $j\neq 0$.  But the latter case gives that $z\in k$, a contradiction. Thus $\tau$ has finite order on $R/Q_1$, but in this case the automorphism induced by $\tau$ preserves a non-constant fibration unless $R/Q_1$ if finite-dimensional.  In the case that $Q_1$ has finite-codimension we have $(0)$ is not rational in $(R/Q_1)[t^{\pm m};\tau]$ since this ring is a finite module over a commutative subalgebra of Krull dimension one.
 \end{proof}

\begin{proof}[Proof of Theorem~\ref{thm:primitive} (b).]  
By the Irving-Small reduction techniques (see \cite[Lemma 8.4.29]{Row2}) we may replace $k$ by its algebraic closure and assume that $k$ is algebraically closed.  If $P$ is a prime ideal of $R[t^{\pm 1};\sigma]$ then since $R[t^{\pm 1};\sigma]$ satisfies the strong Nullstellensatz, we have that primitive ideals are rational.  Thus it suffices to prove rational ideals are primitive.  Let $P$ be a rational ideal of $R[t^{\pm 1};\sigma]$.  Then by Lemma \ref{lem:rat}, either $P$ is maximal, in which case it is primitive and there is nothing to prove, or $P=(R\cap P)R[t^{\pm 1};\sigma]$.  In the latter case, we let $Q_1,\ldots ,Q_s$ denote the minimal prime ideals above $P\cap R$ in $R$.  Then $\sigma$ permutes these primes and acts transitively on this set of primes.  Thus there is some $d$ such that $\tau=\sigma^d$ has the property that $\tau(Q_1)=Q_1$.  We have
that $R/(P\cap R)[t^{\pm 1};\sigma]$ is primitive if and only if $R/Q_1[t^{\pm d};\tau]$ is primitive (see \cite[Corollary 2.2]{LM96}) and we have that 
$Q_1R[t^d;\tau]$ is a rational prime ideal of $R[t^{\pm d},\tau]$ by \cite[Corollary 1.2]{Let89}.  Thus it suffices to prove that $(R/Q_1)[t^{\pm d};\tau]$ is primitive.  Let $X={\rm Spec}(R/Q_1)$.  Then the map induced by $\tau$ on $X$ does not preserve a non-constant fibration since $Q_1R[t^{\pm d};\tau]$ is a rational prime ideal of $R[t^{\pm d};\tau]$.  Thus by Lemma \ref{lem:red}, we have that $X\cong \mathbb{G}_m^d$ for some $d$ and so $R/Q_1\cong k[u_1^{\pm 1},\ldots , u_d^{\pm 1}]$.  Moreover, by \cite{GS-semiabelian} we have that since the map induced by $\tau$ on $X$ does not preserve a non-constant fibration, we then have that there is some $x\in X(k)$ that has a Zariski dense orbit. In particular there is a maximal ideal above $Q_1$ whose bi-infinite $\tau$ orbit has intersection equal to $Q_1$.  Then by Proposition \ref{rem:Jordan} we see that $(R/Q_1)[t^{\pm d},\tau]$ is primitive which gives that $R[t^{\pm 1};\sigma]/P$ is primitive and so $P$ is primitive.  
\end{proof}




\section*{Acknowledgments}
The authors thank Ken Brown for inspiring discussions.


\begin{thebibliography}{BGRS10}

\bibitem[Ame11]{Amerik}
E. Amerik, \emph{Existence of non-preperiodic algebraic points for a rational self-map of infinite order}, Math. Res. Lett. \textbf{18} (2011), 251–--256.

\bibitem[AC08]{A-C}
E. Amerik and F. Campana, \emph{Fibrations m\'eromorphes sur certain vari\'et\'es \`a fibr\'e canonique trivial}, Pure Appl. Math. Quart. \textbf{4} (2008), no.~2, 1-–-37.

\bibitem[BGR17]{dynamical Rosenlicht} 
J. P. Bell, D. Ghioca, and Z. Reichstein, \emph{On a dynamical version of a theorem of Rosenlicht}, Ann. Sc. Norm. Sup. Pisa (5) \textbf{17} (2017), no.~1, 187–--204. 

\bibitem[BGRS17]{3-folds} 
J. P. Bell, D. Ghioca, Z. Reichstein, and M. Satriano, \emph{On the Medvedev-Scanlon conjecture for minimal threefolds of non-negative Kodaira dimension}, New York J. Math. \textbf{23} (2017), 213--–225.

\bibitem[BGT10]{BGT}
J.~P.~Bell, D.~Ghioca, and T.~J.~Tucker, \emph{The dynamical Mordell-Lang problem for \'{e}tale maps}, Amer. J. Math. \textbf{132} (2010), no.~6, 1655-–-1675. 

\bibitem[Ber64]{Ber64}  G.~M.~Bergman, \emph{A ring primitive on the right but not on the left,} Proc. Amer. Math. Soc. \textbf{15} (1964), 473--475.

\bibitem[BCM17]{BCM17}
K. Brown, P. A.A.B. Carvalho, and J. Matczuk, \emph{Simple modules and their essential extensions for skew polynomial rings}, preprint, available online at ArXiv:1705.06596.

\bibitem[Dix77]{Dix77}
J. Dixmier, \emph{Id\'eaux primitifs dans les alg\`ebres enveloppantes,} J. Algebra \textbf{48} (1977), 96--112.

\bibitem[DF17]{DF}
R.~Dujardin and C.~Favre, \emph{The dynamical Manin-Mumford problem for plane polynomial automorphisms}, J. Eur. Math. Soc. (JEMS) \textbf{19} (2017), no.~11, 3421--–3465.

\bibitem[Fak03]{Fakhruddin}
N.~Fakhruddin, \emph{Questions on self maps of algebraic varieties}, J. Ramanujan Math. Soc. \textbf{18} (2003), no.~2, 109--–122.

\bibitem[GH]{GH}
D. Ghioca and F. Hu, \emph{Density of orbits of endomorphisms of commutative linear algebraic groups}, New York J. Math., 14 pages, to appear, available online at https://arxiv.org/pdf/1803.03928.pdf.

\bibitem[GNYa]{GNY-1}
D.~Ghioca, K.~D.~Nguyen, and H.~Ye, \emph{The Dynamical Manin-Mumford Conjecture and the Dynamical Bogomolov Conjecture for split rational maps}, J. Eur. Math. Soc. (JEMS), 25 pages, to appear, available online at https://arxiv.org/pdf/1511.06081.pdf.

\bibitem[GNYb]{GNY-2}
D.~Ghioca, K.~D.~Nguyen, and H.~Ye, \emph{The Dynamical Manin-Mumford and the Dynamical Bogomolov Conjecture for endomorphisms of $(\mathbb{P}^1)^n$}, Compos. Math., 37 pages, to appear, available online at https://arxiv.org/pdf/1705.04873.pdf.

\bibitem[GS17]{GS-abelian}
D.~Ghioca and T.~Scanlon, \emph{Density of orbits of endomorphisms of abelian varieties}, Trans. Amer. Math. Soc. \textbf{369} (2017), no.~1, 447---466.

\bibitem[GS]{GS-semiabelian}
D.~Ghioca and M.~Satriano, \emph{Density of orbits of dominant regular self-maps of semiabelian varieties}, Trans. Amer. Math. Soc., 18 pages, to appear, available online at http://arxiv.org/pdf/1708.06221.pdf.

\bibitem[GTZ08]{GTZ}
D.~Ghioca, T.~J.~Tucker, and S.~Zhang, \emph{Towards a dynamical Manin-Mumford conjecture}, Int. Math. Res. Not. IMRN \textbf{2011}, no. 22, 5109--–5122.

\bibitem[GX]{GX}
D.~Ghioca and J.~Xie, \emph{Algebraic dynamics of skew-linear self-maps}, Proc. Amer. Math. Soc., 16 pages, to appear, available online at http://arxiv.org/pdf/1803.03931.pdf.

\bibitem[Irv79]{Irv79}
R. S. Irving, \emph{Prime Ideals of Ore Extensions over Commutative rings,} J. Algebra \textbf{56} (1979), no. 2, 315--342.

\bibitem[Jor93]{Jor93} D.~A.~Jordan, 
\emph{Primitivity in skew Laurent polynomial rings and related rings}, 
Math. Z. \textbf{213} (1993), no. 3, 353--371. 

\bibitem[KRS05]{KRS}
D.~S.~Keeler, D.~Rogalski and J.~T.~Stafford, \emph{Na\"{i}ve noncommutative blowing up}, 
Duke Math. J. \textbf{126} (2005), no.~3, 491–--546.

\bibitem[Lau84]{Laurent}
M.~Laurent, \emph{\'{E}quations diophantiennes exponentielles}, 
Invent. Math. \textbf{78} (1984), no.~2, 299--–327.

\bibitem[LM96]{LM96} A. Leroy and J. Matczuk, \emph{Primitivity of skew polynomial and skew Laurent polynomial rings,}
Comm. Algebra \textbf{24} (1996), no. 7, 2271--2284. 

\bibitem[Let89]{Let89} E. Letzter, \emph{Primitive ideals in finite extensions of Noetherian rings,} J. London Math. Soc. (2) {\bf 39} (1989), no. 3, 427--435.

\bibitem[McC82]{McC82} J.~C.~McConnell, \emph{The Nullstellensatz and Jacobson properties for rings of differential operators}, J. London Math. Soc. (2) 26 (1982), no. 1, 37--42. 


\bibitem[MS14]{MS-Annals}
A.~Medvedev and T.~Scanlon, \emph{Invariant varieties for polynomial dynamical systems}, Ann. of Math. (2) \textbf{179} (2014), no.~1, 81--–177. 

\bibitem[Moe80]{Moe80} C. Moeglin, \emph{Id\'eaux bilat\`eres des alg\`ebres enveloppantes,} Bull. Soc. Math. France {\bf 108} (1980), 143--186.

\bibitem[RRZ07]{wild automorphism} 
Z.~Reichstein, D.~Rogalski, and J.~J.~Zhang, \emph{Projectively simple rings}, 
Adv. Math. \textbf{203} (2006), no.~2, 365-–-407. 

\bibitem[Row88a]{Row1} L. H. Rowen, \emph{Ring theory. Vol. I.} Pure and Applied Mathematics, 128. Academic Press, Inc., Boston, MA, 1988. 

\bibitem[Row88b]{Row2} L. H. Rowen, \emph{Ring theory. Vol. II.} Pure and Applied Mathematics, 128. Academic Press, Inc., Boston, MA, 1988. 

\bibitem[Zan93]{Zannier}
U.~Zannier, \emph{Some remarks on the $S$-unit equation in function fields}, Acta Arith. \textbf{64} (1993), no.~1, 87-–-98.

\bibitem[Zha06]{Zhang}
S.~Zhang, \emph{Distributions in algebraic dynamics}, Surveys in differential geometry. Vol. X, 381--–430, Surv. Differ. Geom. \textbf{10}, Int. Press, Somerville, MA, 2006.

\end{thebibliography}
\end{document}